\documentclass[10pt, a4paper]{amsart}
\usepackage[dvipsnames]{xcolor}
\usepackage{amsmath, amssymb,amsthm}
\usepackage[extra]{tipa}
\usepackage{lettrine}
\usepackage[Sonny]{fncychap}
\usepackage[english]{babel}
\usepackage{amsmath, amssymb,amsthm}
\usepackage[all]{xy}
\usepackage{tabularx}
\usepackage[applemac]{inputenc}
\usepackage{graphicx} 
\usepackage[dvipsnames]{xcolor}
\usepackage{hyperref}
\usepackage[applemac]{inputenc}
\usepackage{enumitem}
\usepackage{ bbold }
\usepackage{geometry}
\usepackage{prerex}
\usetikzlibrary{fit}
\usepackage{pdflscape}
\usepackage{mathtools}
\DeclarePairedDelimiter\abs{\lvert}{\rvert}

\newcommand\smvee{\raise0.9ex\hbox{$\scriptscriptstyle\vee$}}
\makeatletter
\def\oversortoftilde#1{\mathop{\vbox{\m@th\ialign{##\crcr\noalign{\kern3\p@}%
      \sortoftildefill\crcr\noalign{\kern3\p@\nointerlineskip}%
      $\hfil\displaystyle{#1}\hfil$\crcr}}}\limits}

\def\sortoftildefill{$\m@th \setbox\z@\hbox{$\braceld$}%
  \braceld\leaders\vrule \@height\ht\z@ \@depth\z@\hfill\braceru$}

\makeatother

\usepackage{mathrsfs}
\DeclareMathOperator{\T}{\mathbb{T}}

\DeclareMathOperator{\Pic}{Pic}
\DeclareMathOperator{\Q}{\mathbf{Q}}
\DeclareMathOperator{\Z}{\mathbf{Z}}
\DeclareMathOperator{\C}{\mathbf{C}}
\DeclareMathOperator{\Gal}{Gal}

\DeclareMathOperator{\Hom}{Hom}

\DeclareMathOperator{\Aut}{Aut}

\DeclareMathOperator{\Tr}{Tr}
\DeclareMathOperator{\cusps}{cusps}

\DeclareMathOperator{\SL}{SL}
\DeclareMathOperator{\GL}{GL}
\DeclareMathOperator{\red}{red}

\DeclareMathOperator{\Eis}{Eis}
\newcommand{\cf}{\textit{cf. }}
\newcommand{\ie}{\textit{i.e. }}

\newcommand{\eg}{\textit{eg. }}
\theoremstyle{definition}

\newtheorem{rem}{Remark}[section]

\theoremstyle{plain}
\newtheorem{thm}{Theorem}[section]

\newtheorem{conj}[thm]{Conjecture}
\newtheorem{lem}[thm]{Lemma}

\newtheorem{prop}[thm]{Proposition}

\title{On triple product L-functions and a conjecture of Harris--Venkatesh}
\author{Emmanuel Lecouturier}

\begin{document}
\maketitle

\begin{abstract} 
Harris and Venkatesh made a conjecture relating the derived Hecke operators and the adjoint motivic cohomology in the setting of weight one modular forms. This conjecture was proved under some conditions in the dihedral case by Darmon--Harris--Rotger--Venkatesh. We use a new approach to prove more general cases of the conjecture (up to sign). Our approach relies on Waldspurger's formula for the central value of Rankin L-series and Ichino's formula for the triple product L-function.
\end{abstract}
\tableofcontents

\section{Introduction}\label{Section_introduction}

The aim of this paper is to prove, up to sign, some cases of a conjecture of Harris--Venkatesh on dihedral weight one modular forms. This conjecture was made in \cite{HV}, and proven under some assumptions in \cite{DHRV}. We refer to the introduction of \cite{DHRV} for some background on the conjecture. While our results are similar to those of \cite{DHRV}, they are more general and the methods are quite different. See Remark \ref{rem_difference} below for some discussion  comparing our two results.

Let $N$, $p$ $\geq 5$ be primes such that $p \mid \mid N-1$. Fix embeddings $\overline{\Q} \hookrightarrow \overline{\Q}_p$ and $\overline{\Q} \hookrightarrow \overline{\C}_p$. Let $\T$ (resp. $\T^0$) be the Hecke algebra (over $\Z$) acting on $M_2(\Gamma_0(N))$ (resp. $S_2(\Gamma_0(N))$). Let $I \subset \T$ be the Eisenstein ideal and $\mathfrak{P} = I+(p)$ be the $p$-Eisenstein maximal ideal of $\T$. We denote by $\T_{\mathfrak{P}}$ and $\T_{\mathfrak{P}}^0$ the corresponding $\mathfrak{P}$-adic completions. We let $g_p = \text{rank}_{\Z_p} \T^0_{\mathfrak{P}} \geq 1$. Fix a surjective group homomorphism $\log : (\Z/N\Z)^{\times} \rightarrow \Z/p\Z$. This choice of $\log$ gives rise to a canonical element $\mathfrak{S} \in \Hom(S_2(\Gamma_0(N), \Z_p), \Z/p\Z)$, called the \emph{Shimura class}, which is annihilated by $\mathfrak{P}$ (\cf \cite[\S 3.1]{HV}).

Let us review briefly the context and notation of the conjecture of Harris--Venkatesh \cite{HV}. Let $g \in S_1(\Gamma_1(d))$ be a weight one cuspidal newform of conductor $d$. Assume that $\gcd(Np, d)=1$. We denote by $g^*$ the dual form, \ie the cusp form whose Fourier coefficients at $\infty$ are the complex conjugates of those of $g$. We denote by $G$ the trace of $g(z)g^*(Nz)$ from $S_2(\Gamma_0(Nd))$ to $S_2(\Gamma_0(N))$. Attached to $g$ is an odd Artin representation $\rho_g : \Gal(\overline{\Q}/\Q) \rightarrow \GL_2(\mathcal{O}_g)$, where $\mathcal{O}_g$ is the ring of integers of a certain number field $L_g$ containing the Fourier coefficients of $g$ and contained in a cyclotomic field. We let $R = \mathcal{O}_g[\frac{1}{6}]$. As in \cite[\S 1.2]{DHRV}, one can define a $R$-module $U_g$, which is a certain unit group cut out by the adjoint of $\rho_g$. Note that in \cite{DHRV} the authors invert $N$ in $R$, but that this is not actually needed for the definition of $U_g$, and indeed we want $U_g$ to be independent of $N$. Furthermore, $U_g$ is equipped with a ``reduction modulo $N$'' map 
$$\red_N : U_g \rightarrow (\Z/N\Z)^{\times} \otimes_{\Z} R \text{ .}$$

The Harris--Venkatesh conjecture is as follows:

\begin{conj}[Conjecture 1.1 of \cite{DHRV}]\label{HV_conj}
There exists an integer $m_g \geq  1$ and $u_g \in U_g$ such that, for all primes $N$ and $p$ as above, we have in $R/pR$:
$$m_g\cdot \langle G, \mathfrak{S} \rangle = \log(\red_N(u_g)) \text{ .}$$
\end{conj}

In their recent paper \cite{DHRV}, Darmon--Harris--Rotger--Venkatesh prove this conjecture in certain cases when $g$ is dihedral. Let $K$ be a quadratic field of discriminant $D$, and let $g = \theta_{\psi_1}$ be a weight one theta series associated to a Hecke character $\psi_1$ of $K$. Let $\psi = \frac{\psi_1}{\psi_1'}$, where $\psi_1'$ is the $\Gal(K/\Q)$-conjugate of $\psi_1$. We assume that $g$ is cuspidal, which is equivalent to $\psi_1 \neq \psi_1'$. The main result of \cite{DHRV} is the following:

\begin{thm}[Theorem 1.2 of \cite{DHRV}]\label{main_result_DHRV}
If $K$ is imaginary, assume that $D$ is an odd prime and that $\psi_1$ is unramified. If $K$ is real, assume that $D$ is odd and that the conductor of $\psi_1$ divides $(\sqrt{D})$. Then Conjecture \ref{HV_conj} holds for $g= \theta_{\psi_1}$.
\end{thm}

We prove Conjecture \ref{HV_conj} up to sign, under less restrictives assumptions. The following is our main result.

\begin{thm}\label{main_thm} Assume that $D$ is odd and that $\psi := \frac{\psi_1}{\psi_1'}$ is unramified. Then there exists $u_g\in U_g$ and $m_g \in \Z$ (depending only on $g$) such that for all primes $N$ and $p$ as above, we have, in $R/pR$:
\begin{equation}\label{main_thm_eq}
m_g \cdot \langle G, \mathfrak{S} \rangle = \pm \log(\red_N(u_g)) \text{ .}
\end{equation}
Here, the sign $\pm$ may depend on $N$ and $p$.
\end{thm}

\begin{rem}\label{rem_difference}
Let us discuss the differences between our result and the one of \cite{DHRV}. On the one hand, we only prove the desired identity up to sign, while the result of \cite{DHRV} does not have this ambiguity. On the other hand, our ramification assumption on $\psi_1$ is weaker than the one of Theorem \ref{main_result_DHRV}. Furthermore, we remove the assumption that $D$ is prime in the imaginary case. Let us also point out that \cite{DHRV} does not assume that $p^2$ does not divide $N-1$. We have imposed this condition in order to simplify the commutative algebra of \S \ref{section_abstract}. We do not know whether this may be removed.

Our proof is relatively short compared to that of \cite{DHRV}. The main reason is that we rely on existing results of on special values of L-functions (due to Waldspurger, Gross and Popa), while \cite{DHRV} is more self-contained and uses the theory theta lifts. It seems to us that these two kinds of results use similar techniques (see also the comment at the beginning of \cite[\S 2]{DHRV}). However, the results on special values of L-functions allow us to remove some limitations which arise in the theta lifting method. We expect to be able to remove the assumption that $\psi$ is unramified, since we still have a similar, but more complicated, special L-value formula in this case.
\end{rem}

\begin{rem}\label{thank_venkatesh}
In an earlier version of this paper, we limited ourselves to the case where $g$ is self-dual, \ie $g=g^*$. The reason is that the special L-value formulae involve the square of the absolute value, \eg the left-hand side of (\ref{imaginary_Gross_formula}) is $\abs*{\langle G, f \rangle}^2$. In order to get an identity up to sign, we need to remove this absolute value, \ie replace $\abs*{\langle G, f \rangle}^2$ by $\langle G, f \rangle^2$. It turns out that the phase of the complex number $\langle G, f \rangle$ is easy to compute (it is basically $\chi_g(N)$, where $\chi_g$ is the nebentype of $g$). The trick is to apply the Atkin--Lehner involution $w_N$. We thank Venkatesh for explaining this trick to us.
\end{rem}

Our idea is to go back to the explicit formulation of the conjecture as found in \cite{HV}. Assume for simplicity that $g_p=1$ and that $g$ has coefficients in $\Z$, the general case being more technical but not fundamentally different. The space $S_2(\Gamma_0(N), \Z/p\Z)_{\mathfrak{P}}$ is generated by a single element, namely the Eisenstein series 
\begin{equation}\label{intr_eq_eis_series}
f_0 = \frac{N-1}{24}+ \sum_{n\geq 1} \left(\sum_{d \mid n, \gcd(d,N)=1} d\right)\cdot q^n \text{ (modulo } p \text{) .}
\end{equation}

 Therefore, the projection $G^{\Eis}$ of $G$ on the Eisenstein component $S_2(\Gamma_0(N), \Z_p)_{\mathfrak{P}}$ is of the form
$$G^{\Eis} \equiv a_0 \cdot f_0\text{ (modulo } p \text{)} $$
for some $a_0 \in \Z_p$. We thus get
$$ \langle G, \mathfrak{S} \rangle \equiv a_0\cdot \langle f_0, \mathfrak{S} \rangle = a_0 \cdot \mathcal{M}_0 \text{ (modulo } p \text{)}$$
where $\mathcal{M}_0 = -8\cdot \sum_{k=1}^{\frac{N-1}{2}} k\cdot \log(k) \neq 0$ is \emph{Merel's unit} (its non-vanishing is equivalent to our assumption $g_p=1$). It turns out that we will \emph{not need} the explicit formula for $\mathcal{M}_0$ due to Merel. We need a way to understand the coefficient $a_0$. 

This is where we use L-functions. As recalled in \cite[\S 1.3]{DHRV} (with a missing absolue value), we have the following identity of Harris--Kudla and Ichino:

\begin{equation}\label{intro_HK_identity}
\abs*{\langle G , f \rangle}^2 = C \cdot L(f/K, 1) \cdot L(f/K, \psi, 1)
\end{equation}
where $f$ is any newform in $S_2(\Gamma_0(N))$, $\langle G , f \rangle$ denote the Petersson inner product, and $C$ is a certain non-zero constant, which is essentially a product of local constant terms depending on $f$ and $g$. We will explain in \S \ref{section_Ichino} that, as a consequence of Ichino's formula, we have $C = C_g \cdot C_N$, where $C_g \in \overline{\Q}$ depends only on $g$ and $C_N \in \Q$ depends only on $N$ but is actually independent of $N$ modulo $p$.

 In this paper, we normalize the Petersson product as follows: for any weight $2$ cusp forms $f_1$ and $f_2$ of level $\Gamma_0(N)$, we define 

\begin{equation}\label{petersson_def}
\langle f_1, f_2 \rangle = 8\pi^2 \cdot \iint_{\Gamma_0(N) \backslash \mathfrak{h}} f_1(z) \overline{f_2(z)} dxdy \text{ .}
\end{equation}

We have a decomposition
$$G = \sum_{f} \lambda_f(G) \cdot f$$
where $f$ runs through all the newforms in $S_2(\Gamma_0(N))$. We can view $\lambda_f(G)$ alternatively in $\C$ or in $\overline{\Q}_p$ thanks to our embeddings $\overline{\Q} \hookrightarrow \overline{\Q}_p$ and $\overline{\Q} \hookrightarrow \overline{\C}_p$. We have $\lambda_f(G) = \frac{\langle G , f \rangle}{\langle f , f \rangle}$.

Since $p \mid \mid N-1$, Mazur showed that there is a \emph{unique} newform $F$ (up to conjugation) which is congruent modulo a prime above $p$ to the Eisenstein series $f_0$. We fix such a newform (among its conjugates) in the rest of the paper. Since we are assuming $g_p=1$, we can canonically view $F$ as having coefficients in $\Z_p$. We can thus view $\lambda_F(G)$ as an element of $\Z_p$ and we have 
\begin{equation}\label{intro_congruence_coeff}
\lambda_F(G) \equiv a_0 \text{ (modulo }p\text{) .} 
\end{equation}

By (\ref{intro_HK_identity}), we have:
\begin{equation}\label{intro_key_identity}
\lambda_F(G)^2 = \abs*{\lambda_F(G)}^2= C \cdot \frac{L(F/K, 1) \cdot L(F/K, \psi, 1)}{\langle F, F \rangle^2} \text{ .}
\end{equation}

A general formula of Waldspurger, made more concrete by Gross (resp. Popa) when $K$ is imaginary (resp. real), gives an identity which is roughly of the form:

\begin{equation}\label{intro_eq_waldspurger}
\frac{L(F/K, 1) \cdot L(F/K, \psi, 1)}{\langle F, F \rangle^2} = \frac{\abs*{\langle e_F, e_K' \rangle \cdot \langle e_F', e_{\psi} \rangle}^2}{\langle e_F , e_F' \rangle^2} \text{ ,}
\end{equation}

where $e_F$ (resp. $e_F'$) is an eigenvector corresponding to $F$ in a certain Hecke module $M$ (resp. the dual $M^{\vee}$ of $M$). More precisely, if $K$ is imaginary then $M$ is related to the supersingular module, while if $K$ is real then $M$ is a space of modular symbols. Furtermore, $e_K'$ and $e_{\psi}$ explicit elements of $M$ depending only on $K$ and $\psi$ respectively. These elements are related to some reduction of CM curves modulo $N$ if $K$ is imaginary and to closed Heegner geodesics of discriminant $D$ on $X_0(N)$ if $K$ is real. Let us note that if $K$ is real, then an additional argument if needed for proving (\ref{intro_eq_waldspurger}), namely an expression of the Petersson norm in terms of modular symbols due to Merel \cite{Merel_Manin}.

The right-hand side is an element of $\Z_p$. To understand the numerator modulo $p$, one needs some knowledge of the \emph{higher Eisenstein elements} in $M$. What has been done in \cite{Lecouturier_higher} for the supersingular module and restated in a more suitable form in \cite{DHRV} is enough. One gets that $\langle e_F, e_K \rangle$ is proportional to $\log(\red_N(u_g))$. On the other hand, an abstract argument shows that $\langle e_F , e_F'\rangle$ is proportional to $\mathcal{M}_0$ modulo $p$. This argument does not require the actual explicit knowledge of $\mathcal{M}_0$. Altogether, this proves Theorem \ref{main_thm}.

The plan of this paper is as follows. In Section \ref{section_abstract} we carry out an abstract computation relating $\langle e_F, e_F'\rangle$ to (the higher analogue of) Merel's unit. The main result is Theorem \ref{abstract_main_thm}. In Section \ref{section_Ichino}, we gather some results about special values of L-functions, including Ichino's triple product formula in a classical language (following \cite{Collins}). In Section \ref{section_imaginary} (resp. Section  \ref{section_real}) we deal with the case where $K$ is imaginary (resp. real). 

\subsection*{Acknowledgements}
The author is grateful to Akshay Venkatesh, who pointed out Lemma \ref{phase_computation}, thus allowing to remove the self-dual assumption for $g$.
We also thank Michael Harris and Paul Nelson for providing references regarding Ichino's triple product formula. Part of this work has been carried out while the author was a member at the Institute for Advanced Study (Princeton). We thank the IAS for its hospitality.
This work was funded by Tsinghua University, the Yau Mathematical Sciences Center, the Institute for Advanced Study and a National Natural Science Foundation of China research grant (No. 12050410242).

\section{Higher Eisenstein elements and pairings}\label{section_abstract}

\subsection{Pairings}
In this section, we will use the results and notation of \cite[Section 2]{Lecouturier_higher}. Recall that there is an isomorphism $I/I^2 \otimes \Z_p \xrightarrow{\sim} \Z/p\Z$ sending $T_{\ell}-\ell-1$ to $(\ell-1)\log(\ell)$ for any prime $\ell \neq N$. We fix in what follows a local generator $\eta \in I$ sent to $1$ via this isomorphism.

Let $M$ and $M'$ be $\T$-modules such that $M_{\mathfrak{P}} := M  \otimes_{\T} \T_{\mathfrak{P}}$ and $M_{\mathfrak{P}}'$ are free of rank one over $\T_{\mathfrak{P}}$. Let $e_0$, ..., $e_{g_p}$ (resp. $e_0'$, ..., $e_{g_p}'$) be a system of higher Eisenstein elements of $\overline{M} := M/pM$ (resp. $\overline{M'}$). We can and do assume that for any $i \in \{0, 1, ...,g_p\}$, we have $\eta \cdot e_i = e_{i-1}$ (this is an actual equality in $M/pM$), and similarly for the $e_i'$s.

Assume that we have a $\T$-equivariant bilinear pairing $\langle \cdot , \cdot \rangle : M \times M' \rightarrow \mathbf{Z}$ which induces a \emph{perfect} pairing $M_{\mathfrak{P}} \times M_{\mathfrak{P}}' \rightarrow \Z_p$. The quantity $\langle e_i, e_j' \rangle \in \Z/p\Z$ is well-defined if $i+j\leq g_p$, depends only on $i+j$ and is non-zero if and only if $i+j=g_p$ (\cf \cite[Corollary 2.5]{Lecouturier_higher}).

By \cite[Proposition 1.8]{Emerton}, we have an exact sequence $$0 \rightarrow \Z \cdot T_0 \rightarrow \T \rightarrow \T^0 \rightarrow 0$$ where $T_0 \in \T$ is such that $T_0 - \frac{N-1}{\gcd(N-1, 12)} \in I$. Let $M^0 = M/T_0\cdot M$ and $M_0 = M[T_0]$ (the elements annihilated by $T_0$) (and similarly for $M'$); these are both $\T^0$ modules locally free of rank one at $\mathfrak{P}$. Note that for all $i \in \{0, 1, ..., g_p-1\}$, we have $e_i \in \overline{M}_0$ and that $T_0 \cdot \overline{M} = (\Z/p\Z)\cdot e_0$.

The $\Theta$-correspondence is the canonical $\T$-equivariant homomorphism
$$\Theta : M \otimes_{\T} M' \rightarrow M_2(\Gamma_0(N), \Z[\frac{1}{6}])$$ sending $x \otimes y$ to the modular form whose $q$-expansion at $\infty$ is $ \frac{\gcd(N-1,12)}{24}\cdot \langle x, T_0 y \rangle + \sum_{n\geq 1}  \langle x, T_n y \rangle \cdot q^n$. This induces a $\T_{\mathfrak{P}}$-equivariant isomorphism 
$$\Theta_{\mathfrak{P}} : M_{\mathfrak{P}} \otimes_{\T_{\mathfrak{P}} } M_{\mathfrak{P}}' \xrightarrow{\sim} M_2(\Gamma_0(N), \Z_p)_{\mathfrak{P}} $$
and a $\T_{\mathfrak{P}}^0$-equivariant isomorphism 
$$\Theta_{\mathfrak{P}}^0 : (M_{\mathfrak{P}})^0 \otimes_{\T_{\mathfrak{P}} }(M_{\mathfrak{P}}')_0 \xrightarrow{\sim} S_2(\Gamma_0(N), \Z_p)_{\mathfrak{P}} \text{ ,}$$
which depend on $M$, $M'$ and the pairing $\langle \cdot, \cdot \rangle$. Let $f_0$, ..., $f_{g_p}$ be a system of higher Eisenstein elements of $M_2(\Gamma_0(N), \Z/p\Z)$. We normalise $f_0$ as in (\ref{intr_eq_eis_series}). 

\begin{prop}\label{abstract_theta_higher_eis}
For all $i,j \in \{0, 1, ..., g_p\}$, we have 
$$\Theta_{\mathfrak{P}}(e_i \otimes e_j') \equiv \langle e_0, e_{g_p}' \rangle \cdot f_{i+j-g_p} \text{ (modulo } (\Z/p\Z)\cdot f_0 \oplus ... \oplus (\Z/p\Z)\cdot f_{i+j-g_p-1}  \text{), } $$
with the convention that $f_k=0$ if $k<0$.
\end{prop}
\begin{proof}
If $i+j<g_p$, then $e_i \otimes e_j' = 0$, so the result is true. If $i+j=g_p$, then the element $e_i \otimes e_j'$ is non-zero and annihilated by the Eisenstein ideal. Thus, $\Theta_{\mathfrak{P}}(e_i \otimes e_j') = \lambda\cdot f_0$ for some $\lambda \in \Z/p\Z$. We have $a_1(\lambda\cdot f_0) = \langle e_i, e_j' \rangle = \langle e_0, e_{g_p}' \rangle$, so $\lambda =  \langle e_0, e_{g_p}' \rangle$. The result for $i+j>g_p$ is obtained by induction on $i+j$ using the Hecke property satisfied by the higher Eisenstein elements (\cf \cite[Theorem 2.1]{Lecouturier_higher}).
\end{proof}

\subsection{Eigenvectors}\label{subsection_eigenvectors}

Since we have assumed $p \mid \mid N-1$, the $\Z_p$-algebra $\T_{\mathfrak{P}}^0$ is isomorphic to the valuation ring $\mathcal{O}$ of a totally ramified finite extension $L$ of $\Q_p$. Fix such an isomorphism $\varphi : \T_{\mathfrak{P}}^0 \xrightarrow{\sim} \mathcal{O}$, and let $\pi = \varphi(\eta)$ be a uniformizer of $\mathcal{O}$ ($\eta$ being our fixed local generator of $I$). We denote by $v_{\pi}$ the $\pi$-adic valuation. Since $\mathcal{O}$ is local, the minimal polynomial $P(X) \in \Z_p[X]$ of $\pi$ is distinguished, \ie $P(X) = X^{g_p} + \sum_{i=0}^{g_p-1} a_i X^i$ and $p$ divides $a_i$ for all $i \in \{0, 1, ..., g_p-1\}$.

We use the notation $\mathcal{O}$ for $\T_{\mathfrak{P}}^0$ when we view $\T_{\mathfrak{P}}^0$ as a ``ring of coefficients''. More precisely, let $M_{\mathcal{O}} := M_{\mathfrak{P}}^0 \otimes_{\Z_p} \mathcal{O}$ and $M_{\mathcal{O}} ':= (M_0')_{\mathfrak{P}} \otimes_{\Z_p} \mathcal{O}$. We view $M_{\mathcal{O}}$  as a (left) $\T_{\mathfrak{P}}^0$-module via the action of $\T_{\mathfrak{P}}^0$ on $M_{\mathfrak{P}}^0 $ and as a (left) $\mathcal{O}$-module via the action of $\mathcal{O}$ on itself. Similarly for $M_{\mathcal{O}}'$. We get a perfect $\mathcal{O}$-linear and $\T_{\mathfrak{P}}^0$ equivariant pairing $\langle \cdot , \cdot \rangle : M_{\mathcal{O}} \times M_{\mathcal{O}}' \rightarrow \mathcal{O}$. 

As in the introduction, we fix in the rest of this paper a newform $F \in S_2(\Gamma_0(N), \mathcal{O})$ such that $F \equiv f_0 \text{ (modulo } \pi \text{)}$. 

\begin{prop}\label{abstract_eigenvector}
\begin{enumerate}
\item\label{abstract_eigenvector_i} There is an element $e_F \in M_{\mathcal{O}}$, unique up to multiplication by $\mathcal{O}$, such that for any $T \in \T_{\mathfrak{P}}^0$, we have: $T\cdot e_F = \varphi(T)\cdot e_F$. Similarly for $e_F' \in M_{\mathcal{O}}'$. 

\item\label{abstract_eigenvector_ii}  We have $v_{\pi}(\langle e_F, e_F' \rangle) = v_{\pi}(P'(\pi))$.

\item\label{abstract_eigenvector_iii}  One can choose $e_F$ (resp. $e_F'$) such that $e_F \equiv e_1 \text{ (modulo }\pi\text{)}$ (resp. $e_F' \equiv e_0' \text{ (modulo }\pi\text{)}$). With these choices, we have:
$$\frac{\langle e_F, e_F' \rangle}{P'(\pi)} \equiv \langle e_0, e_{g_p}'\rangle \text{ (modulo }\pi\text{).}$$
\end{enumerate}
\end{prop}
\begin{proof}
Proof of (\ref{abstract_eigenvector_i}). Suppose such an element $e_F$ exists. One can write uniquely $e_F = \sum_{i=0}^{g_p-1} m_i \otimes \pi^i$ for some $m_i \in M_{\mathfrak{P}}^0$. The condition $\eta \cdot e_F = \pi \cdot e_F$ together with the equation $\pi^{g_p} + \sum_{i=0}^{g_p-1} a_i \cdot \pi^i = 0$ implies that for all $i \in \{0, 1, ..., g_p-1\}$ we have $m_i =  \left( \sum_{k=i+1}^{g_p} a_k\cdot \eta^{k-i-1}\right) \cdot m_{g_p-1}$ (where by convention $a_{g_p}=1$). Thus, $e_F$ is unique up to a scalar in $\mathcal{O}$. Conversely, the formula
\begin{equation}\label{abstract_formula_e_F}
e_F = \sum_{i=0}^{g_p-1} \left( \sum_{k=i+1}^{g_p} a_k\cdot \eta^{k-i-1}\right)\cdot m \otimes \pi^i
\end{equation}

 defines the desired element $e_F$ (for any choice of $m$ in $M_{\mathfrak{P}}^0$).
\\
Proof of (\ref{abstract_eigenvector_ii}). For simplicity, let us assume $L/\Q_p$ is Galois and let $G=\Gal(L/\Q_p)$. The general case is identical, up to replacing Galois automorphisms by embeddings $\sigma : L \hookrightarrow \overline{\Q}_p$.

We shall need the following result:

\begin{lem}\label{abstract_lemma_trace}
We have:
$$
\Tr_{L/\Q_p}\left(\frac{\pi^i}{P'(\pi)}\right)= \left\{
    \begin{array}{ll}
        0 & \mbox{if } 0 \leq i < g_p-1 \\
        1 & \mbox{ if } i=g_p-1
    \end{array}
\right.
$$
\end{lem}
\begin{proof}
For $0 \leq i \leq g_p-1$, we have the following identity:

$$\frac{X^i}{P(X)} = \sum_{\sigma \in G} \sigma\left(\frac{\pi^i}{P'(\pi)}\right)\cdot \frac{1}{X-\sigma(\pi)} \text{ .}$$
The lemma then follows from comparing the series expansion in $1/X$.
\end{proof}

Let $m \in M_{\mathfrak{P}}^0$ (resp. $m' \in (M_0')_{\mathfrak{P}}$) be a generator over $\T_{\mathfrak{P}}^0$. Choose $e_F$ and $e_F'$ as in (\ref{abstract_formula_e_F}). 
By Lemma \ref{abstract_lemma_trace}, we get $\Tr_{L/\Q_p}\left(\frac{e_F}{P'(\pi)}\right) = m$, \ie
\begin{equation}\label{abstract_coeff_lambda}
m = \sum_{\sigma \in G} \sigma\left(\frac{1}{P'(\pi)}\right) \cdot \sigma(e_F) \text{ .}
\end{equation}

The family $(\sigma(e_F))_{\sigma \in G}$ is a basis of $M_L := M_{\mathcal{O}} \otimes_{\mathcal{O}} L$ (where $L = \text{Frac}(\mathcal{O})$), and similarly for $(\sigma(e_F'))_{\sigma \in G}$. Since $(1, \eta, ..., \eta^{g_p-1})$ is a $\Z_p$-basis of $\T_{\mathfrak{P}}^0$, we see that $(m, \eta\cdot m, ..., \eta^{g_p-1}\cdot m)$ is a $\mathcal{O}$-basis of $M_{\mathcal{O}}$. The matrix of $(m, \eta\cdot m, ..., \eta^{g_p-1}\cdot m)$ in the $L$-basis $(\sigma(e_F))_{\sigma \in G}$ of $M_L$ is $P = (\sigma(\frac{\pi^i}{P'(\pi)}))_{\sigma \in G,  0 \leq i \leq g_p-1}$. Note that $\det(P)^2= \prod_{\sigma \in G} \sigma(\frac{1}{P'(\pi)})$. We have an equality of matrices
$$ (\langle \eta^i \cdot m, \eta^j \cdot m' \rangle)_{0\leq i,j \leq g_p-1} = P^\intercal \cdot (\langle \sigma(e_F), \tau(e_F') \rangle)_{\sigma, \tau \in G} \cdot P\text{ .}$$
Since the pairing $\langle \cdot , \cdot \rangle : M_{\mathcal{O}} \times M_{\mathcal{O}}' \rightarrow \mathcal{O}$ is perfect, the determinant of the left hand side is a unit. We thus get $$v_{\pi}\left(\prod_{\sigma \in G} \sigma\left( \frac{\langle e_F, e_F' \rangle}{P'(\pi)}\right)\right)=0,$$
which proves that $v_{\pi}(\langle e_F, e_F' \rangle) = v_{\pi}(P'(\pi))$. Let us now compute, $\frac{\langle e_F, e_F' \rangle}{P'(\pi)}$ modulo $\pi$.
\\
Proof of (\ref{abstract_eigenvector_iii}).
Regarding the first assertion, since for all $0 \leq k<g_p$ we have $a_k \equiv 0 \text{ (modulo } p \text{)}$, it suffices to choose $m$ and $m'$ such that $m \equiv e_{g_p} \text{ (modulo }p \text{)}$ and $m' \equiv e_{g_p-1}' \text{ (modulo }p \text{)}$ respectively. The index discrepancy in the higher Eisenstein elements between $m$ and $m'$ comes from the fact that $M_{\mathcal{O}}$ is the maximal cuspidal \emph{quotient}, while $M_{\mathcal{O}}'$ is the maximal cuspidal \emph{subspace}. Thus, in some sense, $e_1$ is the \emph{Eisenstein element} of $M_{\mathcal{O}}$ (and not a higher Eisenstein element).

We have:
\begin{align*}
\Tr_{L/\Q_p}\left(\frac{\langle e_F, e_F'\rangle}{P'(\pi)}\right) &= \sum_{\sigma \in G} \sigma\left( \langle \frac{e_F}{P'(\pi)}, e_F'\rangle \right) 
\\& = \sum_{\sigma \in G} \langle \sigma(\frac{e_F}{P'(\pi)}), \sigma(e_F')\rangle 
\\& = \sum_{\sigma, \tau \in G}  \langle \sigma(\frac{e_F}{P'(\pi)}), \tau(e_F')\rangle 
\\& = \langle \Tr_{L/\Q_p}(\frac{e_F}{P'(\pi)}),  \Tr_{L/\Q_p}(e_F') \rangle \text{ ,}
\end{align*}

where the third equality comes from the fact that $\langle \sigma(e_F), \tau(e_F')\rangle=0$ if $\sigma \neq \tau$. By (\ref{abstract_coeff_lambda}), we have $\Tr_{L/\Q_p}(\frac{e_F}{P'(\pi)})=m\equiv e_{g_p} \text{ (modulo }p \text{)}$. On the other hand, we have $$\Tr_{L/\Q_p}(e_F')  \equiv \left(\sum_{k=1}^{g_p} a_k\cdot \eta^{k-1}\right)\cdot m' \equiv e_0' \text{ (modulo }p \text{).}$$
This concludes the proof of (\ref{abstract_eigenvector_iii}).

\end{proof}

We can finally prove the main result of this section.

\begin{thm}\label{abstract_main_thm}
Let $f \in S_2(\Gamma_0(N), \Z_p)_{\mathfrak{P}}$ and write $f \equiv \sum_{i=0}^{g_p-1} \lambda_i(f) \cdot f_i \text{ (modulo }p \text{)}$ for some $\lambda_i(f) \in \Z/p\Z$. Let $\lambda_F(f)  \in L$ be such that $f = \Tr_{L/\Q_p}(\lambda_F(f) \cdot F)$, \ie $\lambda_F(f)$ is the coefficient of $f$ in $F$. We have $\langle e_F, e_F' \rangle \cdot \lambda_F(f) \in \mathcal{O}$ and 
$$\langle e_F, e_F' \rangle \cdot \lambda_F(f) \equiv \langle e_0, e_{g_p}' \rangle \cdot \lambda_{g_p-1}(f) \text{ (modulo } \pi \text{).}$$
\end{thm}
\begin{proof}
It suffices to check to check this statement for $f = \Theta_{\mathfrak{P}}^0(m \otimes m')$ where $m \in M_{\mathcal{O}}$ and $m' \in M_{\mathcal{O}}'$. Write $m \equiv \sum_{i=1}^{g_p} \lambda_i(m) \cdot e_i \text{ (modulo }p \text{)}$ and $m' \equiv \sum_{i=0}^{g_p-1} \lambda_i(m') \cdot e_i \text{ (modulo }p \text{)}$. By Proposition \ref{abstract_theta_higher_eis}, we have:
$$\Theta_{\mathfrak{P}}^0(m \otimes m') = \langle e_0, e_{g_p}' \rangle \cdot \sum_{1 \leq i \leq g_p \atop 0\leq j \leq g_p-1} \lambda_i(m)\cdot \lambda_j(m')\cdot f_{i+j-g_p} \text{ .}$$
We thus get 
\begin{equation}\label{abstract_proof_main_thm_1}
\lambda_{g_p-1}(f) \equiv \langle e_0, e_{g_p}' \rangle \cdot  \lambda_{g_p}(m)\cdot \lambda_{g_p-1}(m') \text{ (modulo } p \text{).}
\end{equation}

On the other hand, write $m = \Tr_{L/\Q_p}(\lambda_F(m) \cdot e_F)$ and $m' = \Tr_{L/\Q_p}(\lambda_F(m') \cdot e_F')$. By orthogonality of the various conjugates of $e_F$ and $e_F'$, we get: 
$$\Theta_{\mathfrak{P}}^0(m \otimes m') = \Theta_{\mathfrak{P}}^0\left(\Tr_{L/\Q_p}( \lambda_F(m)\cdot \lambda_F(m')\cdot e_F \otimes e_F')\right) \text{ .}$$
We claim that $\Theta_{\mathfrak{P}}^0(e_F \otimes e_F') = \langle e_F, e_F' \rangle \cdot F$. Both sides are proportional since they have the same eigenvalues for the Hecke operators, and the constant of proportionality is given by the first Fourier coefficient of $\Theta_{\mathfrak{P}}^0(e_F \otimes e_F')$, which is $ \langle e_F, e_F' \rangle$ by definition. We thus get:
$$\Theta_{\mathfrak{P}}^0(m \otimes m') = \Tr_{L/\Q_p}\left(\lambda_F(m)\cdot \lambda_F(m')\cdot \langle e_F, e_F' \rangle \cdot F\right) \text{ .}$$
This proves that $\lambda_F(f) = \lambda_F(m)\cdot \lambda_F(m')\cdot \langle e_F, e_F' \rangle$, and so 
\begin{equation}\label{abstract_proof_main_thm_2}
\langle e_F, e_F' \rangle \cdot \lambda_F(f) =  (\langle e_F, e_F' \rangle \cdot \lambda_F(m))\cdot  (\langle e_F, e_F' \rangle \cdot\lambda_F(m')) \text{ .}
\end{equation}
By (\ref{abstract_proof_main_thm_1}) and (\ref{abstract_proof_main_thm_2}), in order to prove Theorem \ref{abstract_main_thm}, it is enough to show that $\langle e_F, e_F' \rangle \cdot \lambda_F(m) \in \mathcal{O}$ and $\langle e_F, e_F' \rangle \cdot \lambda_F(m) \equiv  \langle e_0, e_{g_p}' \rangle \cdot \lambda_{g_p}(m) \text{ (modulo } \pi \text{)}$, and similarly $\langle e_F, e_F' \rangle \cdot \lambda_F(m') \in \mathcal{O}$ and $\langle e_F, e_F' \rangle \cdot \lambda_F(m') \equiv \langle e_0, e_{g_p}' \rangle \cdot \lambda_{g_p-1}(m') \text{ (modulo } \pi \text{)}$. The statement is true for $m \equiv e_{g_p} \text{ (modulo } p \text{)}$: this follows from (\ref{abstract_coeff_lambda}) (from which one gets $\lambda_F(m) = \frac{1}{P'(\pi)}$) and Proposition \ref{abstract_eigenvector} (\ref{abstract_eigenvector_iii}). The statement is then seen to be true for $\eta^i\cdot m$ for any $i \in \{1, ..., g_p-1\}$  (both sides of the equality $\langle e_F, e_F' \rangle \cdot \lambda_F(m) \equiv  \langle e_0, e_{g_p}' \rangle \cdot \lambda_{g_p}(m) \text{ (modulo } \pi \text{)}$ are zero), and thus for any linear combination of such terms.
\end{proof}

\subsection{Higher Eisenstein elements}
In the following paragraphs, we review some results about higher Eisenstein elements that we shall need later. These results are nicely summarised in \cite{DHRV}, so we briefly restate them and generalize them when needed.

\subsubsection{The supersingular module}\label{subsection_supersingular}

The supersingular module is $M = \Z[S]$ where $S$ is the set of isomorphism classes of supersingular elliptic curves over $\overline{\mathbf{F}_{N}}$. It is well-known that elements of $S$ are defined over $\mathbf{F}_{N^2}$. For $E \in S$, we let $w_E = \frac{\Aut(E)}{2} \in \{1,2,3\}$. There is a natural action of $\mathbb{T}$ on $M$. The Atkin--Lehner involution $W_N$ also acts on $M$ (we actually have $W_N=-U_N \in \mathbb{T}$). Explicitely, we have $W_N([E]) = -[E^{(N)}]$ for $[E] \in S$, where $E^{(N)}$ is the elliptic curve obtained from $E$ by raising its coefficients to the power $N$.

There is a natural $\mathbb{T}$-equivariant bilinear pairing $\langle \cdot, \cdot \rangle : M \times M \rightarrow \Z$ defined by 
$$
\langle [E] , [E'] \rangle= \left\{
    \begin{array}{ll}
        0 &  \mbox{ if } E \neq E' \\
        w_E & \mbox{ if } E=E'
    \end{array}
\right.
$$
which is perfect after inverting $6$.

The Eisenstein element $\Sigma_0$ of $M$ is normalised such that
\begin{equation}\label{abstract_equation_sigma0}
\Sigma_0 := \sum_{E \in S} \frac{[E]}{w_E} \text{ .}
\end{equation}
We fix a system of higher Eisenstein elements $\Sigma_0$, $\Sigma_1$, ..., $\Sigma_{g_p}$ of $M/pM$. There is an explicit formula for $\Sigma_1$ (\cf \cite[Equation (91)]{DHRV}),

Recall the notation $g$, $R$, $\psi_1$, $\psi_1'$, $\psi$, $U_g$ and $\red_N$ from the introduction.  We assume in this section that $K$ is imaginary quadratic. Let $\mathcal{H}$ be the Hilbert class field of $K$ and $F$ be a finite abelian extension of $K$ containing $\mathcal{H}$ and such that $\rho_g$ factors through $\Gal(F/\Q)$. We let $H = \Gal(F/K)$, which is an abelian group. Note that $\Gal(K/\Q)$ acts on $H$, but that the action is not in general the map $h \mapsto h^{-1}$ (this is the case if $\psi_1$ is unramified, which is assumed in \cite{DHRV}).
Thus, $\psi_1$, $\psi_1'$ and $\psi$ can be considered as character $H \rightarrow \C^{\times}$. 

As in \cite[\S 5.1]{DHRV}, for any character $\chi : \Pic(\mathcal{O}_K) \rightarrow \Z[\chi]^{\times} $ of the class group of $K$ (\eg $\chi=\psi$ or $\chi = \mathbb{1}$), one can associate an element $[\chi] \in M \otimes \Z[\chi]$ (where $\Z[\chi]$ is the ring generated by the values of $\chi$). The element $[\chi]$ actually depends on a choice of a prime ideal $\mathfrak{N}$ of $F$ above $N$, so in order to keep track of this choice, we shall use the notation $[\chi]_{\mathfrak{N}}$.

The following result will be crucial for us. It is a direct generalization of \cite[Proposition 5.2 and Lemma 5.6]{DHRV}. Actually, the element $u_g$ of \cite[Lemma 5.6]{DHRV} depends on the choice of an ideal $\mathfrak{N}$ above $N$. So does their right hand side, so that the equality written there is correct, but not precise enough for our purposes.

\begin{thm}\label{abstract_unit_supersingular}
There exists $u_g \in U_g$ and $D_g \in \Z$ such that for all $N$ and $p$ as before, we have in $R/pR$:
$$ \abs*{D_g \cdot \langle \Sigma_1, [\psi]_{\mathfrak{N}} \rangle}^2 =\abs{\log(\red_N(u_g))}^2 $$
and
$$\overline{\log(\red_N(u_g))} = -\chi_g(N)^{-1}\cdot \log(\red_N(u_g))$$
where if $x\in R/pR$, we have denoted by $\overline{x}$ the complex conjugate of $x$ and $\abs*{x}^2 = x\overline{x}$.
\end{thm}
\begin{proof}
Let $\mathfrak{N}$ be a prime ideal of $F$ above $N$.
Following \cite[\S 5.1]{DHRV}, one chooses an auxiliary prime $q$ not dividing $D$ and such that $q = \mathfrak{q} \overline{\mathfrak{q}}$ splits in $K$ and $\psi(\mathfrak{q})\neq 1$. As in \cite[Definition 5.1]{DHRV}, one defines a certain elliptic unit $u_{\psi, \mathfrak{q}} \in \mathcal{H}^{\times} \otimes R$ (where as above $\mathcal{H}$ is the Hilbert class field of $K$). It has the property that $\Gal(\mathcal{H}/K)$ acts by multiplication by $\psi^{-1}$ on $u_{\psi, \mathfrak{q}}$. By \cite[Proposition 5.2]{DHRV}, we have
$$(1-\psi(\overline{\mathfrak{q}}))\times \langle \Sigma_1, [\psi]_{\mathfrak{N}} \rangle = -\frac{1}{6}\cdot \log_{\mathfrak{N}}(u_{\psi, \mathfrak{q}}) \text{ ,} $$
where $\log_{\mathfrak{N}}(u_{\psi, \mathfrak{q}}) $ means $\log$ of the reduction of $u_{\psi, \mathfrak{q}}$ modulo $\mathfrak{N}$ (which belongs to $\mathbf{F}_{N^2}^{\times}$ so we extend $\log$ uniquely to $\mathbf{F}_{N^2}^{\times})$. If $\mathfrak{N}' = h(\mathfrak{N})$ is another prime ideal above $N$ (for some $h \in H$), then 

\begin{equation}\label{equation_change_N_log}
\log_{\mathfrak{N}}(u_{\psi, \mathfrak{q}}) = \log_{\mathfrak{N}'}(h(u_{\psi, \mathfrak{q}})) = \psi^{-1}(h)\cdot  \log_{\mathfrak{N}'}(u_{\psi, \mathfrak{q}}) \text{ .}
\end{equation}

Our fixed embedding $\overline{\Q} \hookrightarrow \C$ determines a complex conjugation $\sigma \in \Gal(F/\Q)$. Let $\sigma_{\mathfrak{N}} \in \Gal(F/\Q)$ the Frobenius element corresponding to $\mathfrak{N}$. One modifies slightly the construction of the proof of \cite[Lemma 5.6]{DHRV} as follows. Let $V_g$ be the two dimensional $R$-lattice on which the Artin representation $\rho_g$ acts. We choose a vector $e_1$ of $V_g$ on which $H$ acts via $\psi_1$. Let $e_2 = \sigma(e_1)$. By construction, $(e_1,e_2)$ is a $R$-basis of $V_g$. 

 One can completely describe $\rho_g$ in the basis $(e_1, e_2)$ as follows. If $h\in \Gal(\overline{K}/K)$, then we have
$$\rho_g(h) = \begin{pmatrix} \psi_1(h) & 0 \\ 0 & \psi_1'(h) \end{pmatrix} \text{ ,}$$
where $\psi_1'(h) = \psi_1(\sigma^{-1} h \sigma)$. Furthermore, we have
$$\rho_g(\sigma) = \begin{pmatrix} 0 & 1 \\ 1 & 0 \end{pmatrix} \text{ ,}$$
since $\det(\rho_g(\sigma))=-1$.

Let 
$$u_{g} := u_{\psi, \mathfrak{q}} \otimes \begin{pmatrix} 0 & 1 \\ 0 & 0 \end{pmatrix} + \sigma(u_{\psi, \mathfrak{q}}) \otimes \begin{pmatrix} 0 & 0 \\ 1 & 0 \end{pmatrix} \text{ .}$$
Here, the matrices in the tensor products live in $\text{Ad}(V_g)$, which we identify with $\text{Ad}(V_g)^*$ via the trace pairing. Obviously, $u_g$ is independent of $N$ and $p$.

We claim that $u_g$ is fixed by $\Gal(\overline{\Q}/\Q)$.
If $h \in \Gal(\overline{K}/K)$, then 
\begin{align*}
h(u_g) &= \psi^{-1}(h)\cdot u_{\psi, \mathfrak{q}}\otimes \begin{pmatrix} \psi_1(h) & 0 \\ 0 & \psi_1'(h) \end{pmatrix}  \begin{pmatrix} 0 &1 \\ 0 & 0 \end{pmatrix} \begin{pmatrix} \psi_1(h) & 0 \\ 0 & \psi_1'(h) \end{pmatrix}^{-1} \\& + \psi(h)\cdot \sigma(u_{\psi, \mathfrak{q}})\otimes  \begin{pmatrix} \psi_1(h) & 0 \\ 0 & \psi_1'(h) \end{pmatrix}  \begin{pmatrix} 0 & 0 \\ 1 & 0 \end{pmatrix} \begin{pmatrix} \psi_1(h) & 0 \\ 0 & \psi_1'(h) \end{pmatrix}^{-1} 
\\& = u_g \text{ .}
\end{align*}
We also have:
\begin{align*}
\sigma(u_g) &= \sigma(u_{\psi, \mathfrak{q}})\otimes \begin{pmatrix} 0 & 1 \\ 1 & 0 \end{pmatrix}  \begin{pmatrix} 0 & 1 \\ 0 & 0 \end{pmatrix}  \begin{pmatrix} 0 & 1 \\ 1 & 0 \end{pmatrix}^{-1} \\& +  u_{\psi, \mathfrak{q}} \otimes   \begin{pmatrix} 0 & 1 \\ 1 & 0 \end{pmatrix}  \begin{pmatrix} 0 & 0 \\ 1 & 0 \end{pmatrix}  \begin{pmatrix} 0 & 1 \\ 1 & 0 \end{pmatrix}^{-1} 
\\& = u_g \text{ .}
\end{align*}

We have thus proved that $u_g$ is fixed by $\Gal(\overline{\Q}/\Q)$, and thus by definition belongs to $U_g$. 

Let $h \in H$ such that $\sigma_{\mathfrak{N}} = \sigma \cdot h$. We have 
$$\rho_g(\sigma_{\mathfrak{N}}) = \begin{pmatrix} 0 & 1 \\ 1 & 0 \end{pmatrix} \begin{pmatrix} \psi_1(h) & 0 \\ 0 & \psi_1'(h) \end{pmatrix} = \begin{pmatrix} 0 & \psi_1'(h) \\ \psi_1(h) & 0 \end{pmatrix} \text{ .} $$

 By definition, we have
\begin{align*}
\log(\red_N(u_g)) &= \log_{\mathfrak{N}} \left(\text{Tr}(u_g \cdot \rho_g(\sigma_N)) \right)
\\& =  \log_{\mathfrak{N}}\left( u_{\psi, \mathfrak{q}} \cdot \text{Tr}\left(\begin{pmatrix} 0 & 1 \\ 0 & 0 \end{pmatrix} \cdot \begin{pmatrix} 0 &\psi_1'(h) \\ \psi_1(h) & 0 \end{pmatrix}\right) \right) \\&+  \log_{\mathfrak{N}}\left(\sigma(u_{\psi, \mathfrak{q}}) \cdot  \text{Tr}\left(\begin{pmatrix} 0 & 0 \\ 1 & 0 \end{pmatrix} \cdot \begin{pmatrix} 0 &\psi_1'(h) \\ \psi_1(h) & 0 \end{pmatrix} \right) \right)
\\& =\psi_1(h)\cdot \log_{\mathfrak{N}}(u_{\psi, \mathfrak{q}}) + \psi_1'(h)\cdot \log_{\mathfrak{N}}(\sigma( u_{\psi, \mathfrak{q}})) \text{ .}
\end{align*}
Note that 
$$\log_{\mathfrak{N}}(\sigma(u_{\psi, \mathfrak{q}})) = \log_{\mathfrak{N}}(\sigma_{\mathfrak{N}}h^{-1}(u_{\psi, \mathfrak{q}})) = \psi(h)\cdot   \log_{\mathfrak{N}}(\sigma_{\mathfrak{N}}(u_{\psi, \mathfrak{q}})) = \psi(h)\cdot   \log_{\mathfrak{N}}(u_{\psi, \mathfrak{q}}) \text{ .}$$

We thus have 

\begin{equation}\label{equation_relate_u_psi_g}
\log(\red_N(u_g)) = 2\cdot \psi_1(h) \cdot \log_{\mathfrak{N}}(u_{\psi, \mathfrak{q}})= -12 \cdot \psi_1(h) \cdot (1-\psi(\overline{\mathfrak{q}})) \cdot \langle \Sigma_1, [\psi]_{\mathfrak{N}} \rangle\text{ .}
\end{equation}

One then let $u_g' =  (1-\psi(\mathfrak{q}))  \cdot u_g \in U_g$ and $D_g = -12\cdot \abs*{1-\psi(\mathfrak{q})}^2 \in \Z$, so that 
$$
\log(\red_N(u_g')) = D_g\cdot \psi_1(h) \cdot  \langle \Sigma_1, [\psi]_{\mathfrak{N}} \rangle \text{ .}
$$
We thus get $\abs*{\log(\red_N(u_g'))}^2 = \abs*{D_g\cdot  \langle \Sigma_1, [\psi]_{\mathfrak{N}} \rangle}^2$.

One easily checks from the definition that for any character $\chi$ of the class group of $K$, we have:
\begin{equation}\label{equation_complex_conjug_symbol_chi}
\overline{[\chi]_{\mathfrak{N}} }= -\chi(h) \cdot W_N [\chi]_{\mathfrak{N}} \text{ ,}
\end{equation}
where $h \in H$ is such that $\sigma_{\mathfrak{N}} = \sigma \cdot h$. Thus, we have in $R/pR$:
\begin{align*}
\overline{\log(\red_N(u_g'))} &=  D_g\cdot \psi_1(h)^{-1}\cdot \langle \Sigma_1, \overline{[\psi]_{\mathfrak{N}}}\rangle
\\& = D_g\cdot \psi_1(h)^{-1}\cdot \psi(h)\cdot  \langle \Sigma_1, -W_N[\psi]_{\mathfrak{N}}\rangle 
\\& = D_g\cdot \psi_1(h)^{-1}\cdot \psi(h)\cdot  \langle \Sigma_1, [\psi]_{\mathfrak{N}}\rangle 
\\& = \psi_1(h)^{-2}\cdot \psi(h)\cdot \log(\red_N(u_g'))
\\& = \psi_1(h)^{-1} \cdot \psi_1'(h)^{-1} \cdot \log(\red_N(u_g')) \text{ .}
\end{align*}
We have used the fact that  $\langle \Sigma_1, -W_N[\psi]_{\mathfrak{N}}\rangle =   \langle -W_N\Sigma_1, [\psi]_{\mathfrak{N}}\rangle = \langle \Sigma_1, [\psi]_{\mathfrak{N}}\rangle $.
Note that we have 
\begin{equation}\label{eq_chi_N_h}
\chi_g(N) = \det(\rho_g(\sigma_{\mathfrak{N}})) =  \det(\rho_g(\sigma \cdot h)) = - \det(\rho_g(h)) = -\psi_1(h)\cdot \psi_1'(h) \text{ .}
\end{equation}
We have thus proved:
$$\overline{\log(\red_N(u_g'))} = - \chi_g(N)^{-1} \cdot \log(\red_N(u_g')) \text{ .}$$
Thus, $u_g'$ satisfies the two conditions of Theorem \ref{abstract_unit_supersingular}.
\end{proof}

\subsubsection{Modular symbols}
We refer to \cite[\S 4.3, 4.4]{DHRV} for the results of this paragraph.
We let $M^+ = H^1(X_0(N), \cusps, \Z_p)^+$ and $M^- = H^1(Y_0(N), \Z_p)^-$. Here, we are using singular cohomology and the sign corresponds to the eigenvalue for the action of the complex conjugation. There is a natural action of $\mathbb{T}$ (and thus $W_N$) on $M^+$ and $M^-$. There is a natural perfect $\mathbb{T}$-equivariant bilinear pairing $\langle \cdot, \cdot \rangle : M^+ \times M^- \rightarrow \Z_p$ given Poincar\'e duality. Note that, since $p>3$, one can ignore elliptic points and we get a canonical isomorphism
$M^{-} \simeq \Hom(\Gamma_0(N), \Z_p)^-$. Furthermore, the cuspidal quotient $(M^+)^0$ is equal to $H^1(X_0(N), \Z_p)^+$, which can be identified with the subspace of group homomorphism $\varphi \in \Hom(\Gamma_0(N), \Z_p)^+$ vanishing on the parabolic elements. We fix a system of higher Eisenstein elements $\kappa_0^{\pm}$, $\kappa_1^{\pm}$, ..., $\kappa_{g_p}^{\pm}$ in $M^{\pm}/pM^{\pm}$ as in \cite[\S 4.3, 4.4]{DHRV}. In particular, we have the following formula:

$$\kappa_1^+ \begin{pmatrix} a& b \\ c&d \end{pmatrix} = \log(a) $$
for any $\begin{pmatrix} a& b \\ c&d \end{pmatrix}  \in \Gamma_0(N)$.
Furthermore, $\kappa_0^-$ is the $N$-stabilization of the Dedekind--Rademacher homomorphism.

\subsubsection{Modular forms}

We let $\mathcal{M} = M_2(\Gamma_0(N), \Z_p)$. Recall that we have fixed a system of higher Eisenstein elements $f_0$, $f_1$, ..., $f_{g_p}$ of $\mathcal{M}/p\mathcal{M}$, where $f_0$ is the weight $2$ Eisenstein series modulo $p$ defined by (\ref{intr_eq_eis_series}). Let us now consider the dual $\mathcal{M}^{\vee} = \Hom(\mathcal{M}, \Z_p)$. We normalise the Eisenstein element $\mathfrak{S}_0$ of $\mathcal{M}^{\vee}$ by the formula
$$\mathfrak{S}_0(f) = a_0(f)\text{ .}$$
We denote by $\mathfrak{S}_0$, $\mathfrak{S}_1$, ..., $\mathfrak{S}_{g_p}$ a system of higher Eisenstein elements of $\overline{\mathcal{M}}^{\vee}$. Note that $\mathfrak{S}_1$ is only well-defined modulo $\mathfrak{S}_0$, but that the image of $\mathfrak{S}_1$ in $\Hom(S_2(\Gamma_0(N), \Z/p\Z), \Z/p\Z)$ is well-defined. By \cite[Theorem 4.9]{DHRV}, this image is simply the Shimura class $\mathfrak{S}$ (viewed in coherent cohomology, \cf \cite[\S 1.1]{DHRV}) of the introduction.

\subsection{Compatibility between pairings}

Since the various pairings involving the Hecke modules above will appear in our computations, it is important to know that they are essentially the same:

\begin{thm}\label{abstract_comparison_pairings}
For all $i,j \in \{0, 1, ..., g_p \}$ with $i+j=g_p$, we have:
$$\langle \Sigma_i, \Sigma_j \rangle =2  \cdot \langle f_i, \mathfrak{S}_j \rangle = -\frac{1}{12} \langle \kappa_i^+, \kappa_j^- \rangle \text{ .}$$

\end{thm}
\begin{proof}
This follows from \cite[Corollary 6.3]{Lecouturier_higher}. Let us explain the factor $-\frac{1}{12}$. The minus sign comes from the fact that the present normalisation of $\kappa_1^+$ is the opposite of that of \cite[Theorem \S 4.2]{Lecouturier_higher} (we chose this normalisation here to be consistent with the one of \cite{DHRV}). The factor $\frac{1}{12}$ comes from our different normalisation of $\kappa_0^-$ (\cf the sentence after \cite[Theorem 1.9]{Lecouturier_higher}).
\end{proof}

\section{Proof of the main theorem}
We are now ready to prove Theorem \ref{main_thm}. We keep the notation of the previous sections. As in \S \ref{subsection_eigenvectors}, we fix a newform $F \in S_2(\Gamma_0(N), \mathcal{O})$ congruent to $f_0$ modulo $\pi$. We are going to evaluate explicitly both sides of (\ref{intro_key_identity}) (appropriately renormalised) modulo $\pi$. 

\subsection{Special values of L-functions}\label{section_Ichino}

We first make precise the constant $C$ of the triple product formula (\ref{intro_HK_identity}). Ichino \cite{Ichino} expressed $C$ as a product of local terms, which were computed explicitely in some cases by Woodbury. The passage from the automorphic formulation of Ichino's formula to an entirely classical one has been done \cite{Collins}. We now recall Collin's computation. We warn the reader that Collins' normalization of the Petersson product is $\frac{3}{[\SL_2(\Z):\Gamma_0(N)]\cdot 8\pi^3}$ times our normalization (\ref{petersson_def}). Let $f$ be a newform in $S_2(\Gamma_0(N))$ and let $g$ be a newform in $S_1(\Gamma_1(d))$ as in the introduction (in particular, we assume that $\gcd(N,d)=1$).
We have, by \cite[Theorem 2.1.2]{Collins}:

$$
\abs*{\langle g(z)g^*(Nz), f(z)\rangle }^2 = \frac{[\SL_2(\Z) : \Gamma_0(N)]^2}{N} \cdot L(f,g,g^*, 1) \cdot \prod_{q \mid Nd\text{, } q \text{ prime}} \mathcal{E}_qI_q^* \text{ ,}
$$
where for each prime $q$ dividing $Nd$, $\mathcal{E}_qI_q^*$ is a non-zero local factor defined by Ichino and recalled in \cite{Collins}. By \cite[Corollary 2.2.2]{Collins}, we have $\mathcal{E}_NI_N^* = \frac{N}{(N+1)^2}$ (note that $\frac{[\SL_2(\Z) : \Gamma_0(N)]^2}{N} = \frac{(N+1)^2}{N}$ cancels with $\mathcal{E}_NI_N^*$). We let $C_g = \prod_{q \mid d\text{, } q \text{ prime}} \mathcal{E}_qI_q^*$. We shall not need the exact value of $C_g$ (which, to our knowledge, has not been computed in the litterature). Instead, the following result will be enough for our purposes.

\begin{prop}\label{prop_C_g_square}
If $g = \theta_{\psi_1}$ is as in Theorem \ref{main_thm}, then $C_g \cdot \frac{1}{\abs*{D}}$ is a square in $L_g^{\times}$.
\end{prop}
The proof of this proposition will require some preliminaries, which will also be used later in the paper. 

First, we relate $\langle g(z)g^*(Nz), f(z)\rangle ^2$ to $\abs*{\langle g(z)g^*(Nz), f(z)\rangle }^2$. We denote by a bar the complex conjugation acting on $\C$. The following result was indicated to us by Venkatesh.
\begin{lem}\label{phase_computation}
We have 
$$\overline{\langle g(z)g^*(Nz), f(z)\rangle} = \epsilon_N(f)\cdot \chi_g(N)^{-1}\cdot  \langle g(z)g^*(Nz), f(z)\rangle$$
where $\chi_g : (\Z/d\Z)^{\times} \rightarrow \C^{\times}$ is the Nebentype character of $g$ and $\epsilon_N(f) \in \{1,-1\}$ is the Atkin--Lehner eigenvalue of $f$ (at the prime $N$).
\end{lem}
\begin{proof}
If $\Gamma$ is a congruence subgroup, let $X(\Gamma)$ be the corresponding (compact) modular curve. We denote by $\pi_1, \pi_2 : X(\Gamma_0(N) \cap \Gamma_1(d)) \rightarrow X(\Gamma_1(d))$ the two degeneracy maps given by $\pi_1(z) = z$ and $\pi_2(z) = Nz$ on the upper-half plane. Let $W_N$ be the index $N$ Atkin--Lehner involution and $\langle N \rangle$ be the $N$th diamond operator. These are automorphisms of both $X(\Gamma_0(N) \cap \Gamma_1(d))$ and $X(\Gamma_1(d))$ (which we denote by the same letter, as no possible confusion will arise). 
Using the moduli interpretation, we easily check that $W_N^2 = \langle N \rangle$ and $\pi_1 \circ W_N = \langle N \rangle \circ \pi_2$ (note that $W_N$ is not an involution in general). We thus get 
$$\pi_2 \circ W_N = \langle N \rangle^{-1} \circ \pi_1 \circ W_N^2 = \langle N \rangle^{-1} \circ \pi_1 \circ \langle N \rangle = \pi_1 \text{ .}$$

We denote by an upper star the maps induced on the spaces of modular forms (of any weight). We have
\begin{align*}
W_N^*(g(z)g^*(Nz)) &= W_N^*(\pi_1^*(g) \cdot \pi_2^*(g^*)) \\& = (\pi_1 \circ W_N)^*(g) \cdot (\pi_2 \circ W_N)^*(g^*) \\& = 
((\pi_2^* \circ \langle N \rangle^*) (g))\cdot \pi_1^*(g^*) \\& = \chi_g(N)\cdot g(Nz)g^*(z) \text{ .}
\end{align*}
Since $f$ has level $\Gamma_0(N)$, its Fourier coefficients are totally real. We thus have $$\overline{\langle g(z)g^*(Nz), f(z)\rangle} = \langle g^*(z)g(Nz), f(z)\rangle \text{ .}$$
Thus, we have 
\begin{align*}
\overline{\langle g(z)g^*(Nz), f(z)\rangle} &= \langle g^*(z)g(Nz), f(z)\rangle 
\\&= \chi_g(N)^{-1} \cdot \langle W_N^*(g(z)g^*(Nz)), f(z) \rangle
\\& = \chi_g(N)^{-1} \cdot \langle g(z)g^*(Nz), W_N^*(f)\rangle
\\& = \chi_g(N)^{-1} \epsilon_N(f) \cdot \langle g(z)g^*(Nz), f(z) \rangle
\end{align*}
\end{proof}

In the remainder of this section, we assume that $g = \theta_{\psi_1}$ and we fix a newform $f \in S_2(\Gamma_0(N))$. Recall that $G(z) = g(z)g^*(Nz)$. We have $\langle g(z)g^*(Nz), f \rangle = \lambda_f(G)\cdot \langle f , f\rangle$, where $\lambda_f(G)$ belongs to the compositum $K_f\cdot L_g$ of $L_g$ and the Hecke field $K_f$ of $f$. Thus, we have:

\begin{align*}
\lambda_f(G)^2 &= \frac{\langle g(z)g^*(Nz), f \rangle^2}{\langle f, f \rangle^2} 
\\& =  \frac{\chi_g(N)\cdot \epsilon_N(f) \cdot \abs*{\langle g(z)g^*(Nz), f \rangle}^2}{\langle f, f \rangle^2} 
\end{align*}
 By the Artin formalism, we have $L(f,g,g^*, 1) = L(f/K, 1) \cdot L(f/K, \psi, 1)$ where we recall that $\psi = \frac{\psi_1}{\psi_1'}$ ($\psi_1'$ being the conjugate character of $\psi_1$).  By Ichino's formula, we get

\begin{equation}\label{key_eq_lambda_f_ichino}
\lambda_f(G)^2 = C_g\cdot \chi_g(N)\cdot \epsilon_N(f) \cdot \frac{L(f/K, 1) \cdot L(f/K, \psi, 1)}{\langle f, f \rangle^2}
\end{equation}

We now sketch the proof of Proposition \ref{prop_C_g_square}. Recall the notation $\mathcal{O}_g$, $L_g$, $\psi_1$, $\psi_1'$ and $\psi$ from the introduction. 
\begin{proof}
Let us first assume that $K$ is imaginary, \ie $D<0$. Let $u(K) = \# (\mathcal{O}_K^{\times}/\pm1) \in \{1,2,3\}$. Gross gave an explicit formula for the special value $L(f, \chi, 1)$ if $\chi$ is a character of the class group of $K$ (\cf \cite[Proposition 11.2]{Gross}). The formula is as follows:
\begin{equation}\label{imaginary_Gross_formula}
\frac{L(f/K, \chi, 1)}{\langle f, f \rangle} = \frac{1}{u(K)^2\cdot \sqrt{-D}}\cdot \frac{\abs*{\langle [\chi]_{\mathfrak{N}}, e_f' \rangle}^2}{\langle e_f', e_f'\rangle} \text{ .}
\end{equation}
where $e_f'$ is any non-zero element of $K_f[S]^0$ (the augmentation subgroup of the supersingular module with coefficients in $K_f$) such that $T_{n}(e_f') = a_{n}(f)\cdot e_f'$ for all $n\geq 1$. Here, $T_n$ is the $n$th Hecke operator and $a_n(f)$ is the $n$th Fourier coefficient of $f$ at $\infty$. 

Recall that $ [\chi]_{\mathfrak{N}} \in \Z[\chi][S]$ was defined in \S \ref{subsection_supersingular}, and depends on a choice of a prime ideal $\mathfrak{N}$ above $N$ in $F$. Recall that by (\ref{equation_complex_conjug_symbol_chi}) we have:
$$
\overline{[\chi]_{\mathfrak{N}} }= -\chi(h) \cdot W_N[\chi]_{\mathfrak{N}} \text{ ,}
$$
where $h \in \Gal(F/\Q)$ is such that $\sigma_{\mathfrak{N}} = \sigma \cdot h$ (using the notation of the proof of Theorem \ref{abstract_unit_supersingular}). As in (\ref{eq_chi_N_h}), we have $\chi_g(N) = -\psi_1(h)\cdot \psi_1'(h)$.

Since our form $f$ is self-dual, we get
\begin{align*}
\overline{\langle [\chi]_{\mathfrak{N}}, e_f' \rangle} &= {\langle \overline{[\chi]_{\mathfrak{N}}}, e_f' \rangle} 
\\& = -\chi(h)\cdot \langle W_N[\chi]_{\mathfrak{N}}, e_f' \rangle 
\\& = -\chi(h)\cdot \langle [\chi]_{\mathfrak{N}}, W_Ne_f' \rangle 
\\&= -\epsilon_N(f)\cdot \chi(h)\cdot \langle [\chi]_{\mathfrak{N}}, e_f' \rangle \text{ .}
\end{align*}
Thus, Gross' formula can be rewritten as
$$\frac{L(f/K, \chi, 1)}{\langle f, f \rangle} = -\frac{\epsilon_N(f)\cdot \chi(h)}{u(K)^2\cdot \sqrt{-D}}\cdot \frac{\langle [\chi]_{\mathfrak{N}}, e_f' \rangle^2}{\langle e_f', e_f'\rangle} $$

Applying this formula for $\chi=\mathbb{1}$ and $\chi = \psi$, we get:
$$  \frac{L(f/K, 1) \cdot L(f/K, \psi, 1)}{\langle f, f \rangle^2} = -\frac{\psi(h)}{u(K)^4\cdot D}\cdot \left(\frac{\langle [\mathbb{1}]_{\mathfrak{N}}, e_f' \rangle\cdot \langle [\psi]_{\mathfrak{N}}, e_f' \rangle}{\langle e_f', e_f' \rangle} \right)^2$$
Using (\ref{key_eq_lambda_f_ichino}) and the fact that $\psi(h)\cdot \chi_g(N) = -\psi_1(h)^2$, we get:

\begin{equation}\label{lambda_G_square_eq}
\lambda_f(G)^2 = C_g\cdot \frac{1}{D}\cdot \frac{\epsilon_N(f)\cdot \psi_1(h)^2}{u(K)^4} \cdot \left(\frac{\langle [\mathbb{1}]_{\mathfrak{N}}, e_f' \rangle\cdot \langle [\psi]_{\mathfrak{N}}, e_f' \rangle}{\langle e_f', e_f' \rangle} \right)^2 \text{ .}
\end{equation}
We apply (\ref{lambda_G_square_eq}) in the case where $f$ is the newform of $S_2(\Gamma_0(11))$ (for which $K_f = \Q$ and $\epsilon_{11}(f)=-1$). We get that $-C_g\cdot \frac{1}{D}$ is a square in $L_g^{\times}$, which concludes the proof Proposition \ref{prop_C_g_square} when $K$ is imaginary.

The proof when $K$ is real is very similar, but uses Popa's formula for the special value $L(f/K, \chi, 1)$ instead of Gross' formula. See \S \ref{section_real}, and especially Remark \ref{rem_proof_square_real}, for more details.
\end{proof}

\subsection{Proof of the main theorem: the imaginary quadratic case}\label{section_imaginary}

Assume in this paragraph that $K$ is imaginary. Recall that $D$ is the discriminant of $K$ and that $u(K) = \# (\mathcal{O}_K^{\times}/\pm1) \in \{1,2,3\}$. Let $h(K)$ be the class number of $K$. Let us recall our basic assumptions: $D$ is odd and $g = \theta_{\psi_1}$ for some Hecke character $\psi_1$ of $K$ such that $\psi := \frac{\psi_1}{\psi_1'}$ is unramified. 

We apply the result of \S \ref{subsection_eigenvectors} for $M=M' = \Z[S]$ (the supersingular module). There is a subtlety here: even though $M=M'$, but $M_{\mathcal{O}} \neq M_{\mathcal{O}}'$ and $e_F \neq e_F'$. This is due to the fact that $M_{\mathcal{O}}$ corresponds to the cuspidal \emph{quotient}, while $M_{\mathcal{O}}'$ corresponds to the cuspidal \emph{subspace}. The canonical map $M_0 \rightarrow M^0$ has image $I \cdot M^0$. This induces a canonical map $\varphi : M_{\mathcal{O}}' \rightarrow M_{\mathcal{O}}$, and we have $\varphi(e_F') = \eta\cdot e_F = \pi\cdot e_F$. This can be seen for instance using (\ref{abstract_formula_e_F}) and noting that $m'$ can be chosen such that $\varphi(m') = \eta \cdot m$.

By Theorem \ref{abstract_main_thm}, we have:
\begin{equation}\label{imaginary_lambda_G}
\langle e_F , e_F' \rangle \cdot \lambda_F(G) \equiv \langle \Sigma_0, \Sigma_{g_p} \rangle \cdot \lambda_{g_p-1}(G) \text{ (modulo }\pi\text{).}
\end{equation}
Let us now analyse the right-hand side of (\ref{key_eq_lambda_f_ichino}). Gross' formula (\ref{imaginary_Gross_formula}) gives
$$
\frac{L(F/K, \chi, 1)}{\langle F, F \rangle} = \frac{1}{u(K)^2\cdot \sqrt{-D}}\cdot \frac{\abs*{\langle [\chi], e_F' \rangle}^2}{\langle e_F', e_F'\rangle} \text{ .}
$$
Note that the denominator of the right hand side is \emph{not} $\langle e_F, e_F'\rangle$. Since $\varphi(e_F') = \pi\cdot e_F $, we get $\langle e_F', e_F'\rangle = \pi \cdot \langle e_F, e_F'\rangle$.

Let $\mathbb{1}$ be the trivial character. Since the degree of $[\mathbb{1}]$ is $h(K)$ and $e_F'\equiv \Sigma_0 \text{ (modulo } \pi \text{)}$, we have $\langle e_F', [\mathbb{1}] \rangle \equiv h(K) \text{ (modulo }\pi\text{)}$. Since $e_F \equiv \Sigma_1  \text{ (modulo }\pi\text{)}$, Theorem \ref{abstract_unit_supersingular}, shows that there exists a constant $D_g \in \Z$ and $u_g \in U_g$ such that $\abs*{D_g \cdot \langle [\psi], e_F \rangle}^2 \equiv \abs*{\log(\red_N(u_g))}^2  \text{ (modulo }\pi\text{)}$. Since $\varphi(e_F') = \pi \cdot e_F$, we get $$\abs*{D_g\cdot \langle [\psi], e_F'\rangle}^2 \equiv \pi^2\cdot  \abs*{\log(\red_N(u_g))}^2  \text{ (modulo }\pi^3\text{)} \text{ .}$$

Combining all of the above, we get:

$$
D_g^2\cdot \langle e_F, e_F' \rangle^2 \cdot \frac{L(F/K, 1)\cdot L(F/K, \psi, 1)}{\langle F, F \rangle^2}  \equiv \frac{1}{u(K)^4\cdot \abs*{D}}\cdot (h(K))^2  \cdot \abs*{\log(\red_N(u_g))}^2 \text{ (modulo }\pi\text{).}
$$

By Theorem \ref{abstract_unit_supersingular}, we have
\begin{align*}
\abs*{\log(\red_N(u_g))}^2 &= \log(\red_N(u_g)) \cdot \overline{\log(\red_N(u_g))}
\\& =  - \chi_g(N)^{-1} \cdot  \log(\red_N(u_g))^2
\end{align*}
so we get

\begin{equation}\label{imaginary_main_formula_L}
D_g^2\cdot \langle e_F, e_F' \rangle^2 \cdot \frac{L(F/K, 1)\cdot L(F/K, \psi, 1)}{\langle F, F \rangle^2}  \equiv -\frac{\chi_g(N)^{-1}}{u(K)^4\cdot \abs*{D}}\cdot (h(K))^2  \cdot \log(\red_N(u_g))^2 \text{ (modulo }\pi\text{).}
\end{equation}

Since $F$ is Eisenstein modulo $p$, we have $W_N^*(F) = -U_N(F) = -F$ (note that $U_N=-W_N$ on $X_0(N)$ as $N$ is prime). By (\ref{key_eq_lambda_f_ichino}), we get:

\begin{equation}\label{collins_imaginary_eq}
\lambda_F(G)^2 = -C_g \cdot \chi_g(N) \cdot \frac{L(F/K, 1)\cdot L(F/K, \psi, 1)}{\langle F, F \rangle^2} \text{ .}
\end{equation}

By Proposition \ref{prop_C_g_square}, one can write $C_g =\abs{D}\cdot \frac{A_g^2}{B_g^2}$ for some $B_g \in \Z$ and $A_g \in \mathcal{O}_g$ depending only on $g$. Combining (\ref{imaginary_lambda_G}), (\ref{imaginary_main_formula_L}) and (\ref{collins_imaginary_eq}),  we conclude that:
$$B_g^2\cdot D_g^2 \cdot \left( \langle \Sigma_0, \Sigma_{g_p} \rangle \cdot \lambda_{g_p-1}(G) \right)^2 \equiv \frac{A_g^2}{u(K)^4}\cdot (h(K))^2  \cdot  \log(\red_N(u_g))^2 \text{ (modulo }\pi\text{).} $$
By Theorem \ref{abstract_comparison_pairings}, we have $\langle \Sigma_0, \Sigma_{g_p} \rangle = 2\cdot\langle f_{g_p-1}, \mathfrak{S}_1 \rangle$. Finally, notice that $\langle G, \mathfrak{S} \rangle = \lambda_{g_p-1}(G) \cdot \langle f_{g_p-1},\mathfrak{S}_1 \rangle$. 
In conclusion, we have proven that there exists $M_g \in R$ and $u_g \in U_g$ depending only on $g$ such that $M_g^2 \cdot \langle G , \mathfrak{S} \rangle^2 = \log(\red_N(u_g))^2$ in $R/pR$ for all $N$ and $p$ as before. This proves Theorem \ref{main_thm} when $K$ is imaginary.

\subsection{Proof of the main theorem: the real quadratic case}\label{section_real}
Assume in this paragraph that $K$ is real. Let $D$ be the discriminant of $K$, $h(K)$ the narrow class number and $u(K)$ be a fundamental unit of $\mathcal{O}_K^{\times}$. We fix a prime $\mathfrak{N}$ of $\mathcal{O}_K$ dividing $N$. Let $\mathcal{C}(K)$ be the narrow class group of $\mathcal{O}_K$. Let us recall our basic assumptions: $D$ is odd and $g = \theta_{\psi_1}$ for some Hecke character $\psi_1$ of $K$ such that $\psi := \frac{\psi_1}{\psi_1'}$ is unramified. We denote by $\chi_g : (\Z/d\Z)^{\times} \rightarrow \C^{\times}$ the Nebentype character of $g$. The existence of our unramified character $\psi$ of signature at infinity $(-1,-1)$ implies that the narrow class group is strictly bigger than the (usual) class group, \ie the norm of $u(K)$ is $1$. 

We apply the result of \S \ref{subsection_eigenvectors} for $M=M^+ = H^1(X_0(N), \cusps, \Z_p)^+$ and $M' = M^- = H^1(Y_0(N), \Z_p)^-$. One can consider our eigenform $F$ as having coefficients in $\C$ via our embeddings $\overline{\Q} \hookrightarrow \overline{\Q}_p$ and $\overline{\Q} \hookrightarrow \C$. We get a ``period homomorphism'' $\xi_F : H_1(X_0(N), \Z)\rightarrow \C$ given by $\xi_F(\gamma) = \int_{\gamma} 2i\pi F(z)dz$ for cycles $\gamma \in  H_1(X_0(N), \Z)$. One can write $\xi_F = \xi_F^+ + i\cdot \xi_F^-$, where $\xi_F^{\pm} : H_1(X_0(N), \Z)^{\pm} \rightarrow \mathbf{R}$. Actually, one can extend $\xi_F^+$ to a homomorphism $H_1(X_0(N), \cusps, \Z)^+ \rightarrow \mathbf{R}$ and factor $\xi_F^-$ as a homomorphism $H_1(Y_0(N), \Z)^- \rightarrow \mathbf{R}$. One can then view $\xi_F^+$ and $\xi_F^-$ as elements of $H^1(X_0(N), \cusps, \mathbf{R})^+$ and $H^1(Y_0(N), \mathbf{R})^-$ respectively. 

 There exist periods $\Omega_F^+, \Omega_F^- \in \C^{\times}$, unique up to an element of $\overline{\Q}^{\times}$, such that $\frac{\xi_F^{\pm}}{\Omega_F^{\pm}}$ takes values in $\overline{\Q}$. One can choose the periods so that $e_F = \frac{\xi_F^{+}}{\Omega_F^{+}}$ and $e_F'=  \frac{\xi_F^{-}}{\Omega_F^-}$. In that way, $\Omega_F^{+}$ and $\Omega_F^{-}$ are well-defined up to an element of $1+\pi\cdot \mathcal{O}$.

\begin{prop}\label{real_periods_prop}
One has $\frac{\langle F , F \rangle}{\Omega_F^+ \cdot \Omega_F^-} = 2\cdot \langle e_F, e_F' \rangle$.
\end{prop}
\begin{proof} This is essentially due to Merel, who related $\langle F , F \rangle$ to $\xi_F$ (\cf \cite[Corollaire 4]{Merel_Manin}). One can deduce that $-i\cdot \langle F , F \rangle = \langle  \xi_F, \overline{\xi_F} \rangle$ (for any cusp form $F$). A detailed proof can be found in \cite[Theorem 2.16 and Proposition 2.17]{Lecouturier_mixed} (where we warn the reader that the Petersson product is differently normalised than the one used in the present paper). Finally, note that $ \langle  \xi_F, \overline{\xi_F} \rangle = -2i\cdot \langle \xi_F^+, \xi_F^- \rangle$.

\end{proof}

We refer to \cite[\S 3.1]{DHRV} for the formalism of Heegner cycles. The choice of our prime ideal $\mathfrak{N}$ above $N$ in $K$ corresponds to a choice of $\delta_N \in \Z$ (uniquely determined modulo $N$) such that $\delta_N^2 \equiv D \text{ (modulo }N\text{)}$. For any $[I] \in \mathcal{C}(K)$, one can associate a closed geodesic $\gamma_I$ in the modular curve $X_0(N)$. The definition of $\gamma_I$ actually depends on the choice of $\mathfrak{N}$, but we don't make it explicit for the ease of notation. We easily see that  $\gamma_{\alpha \cdot I}$ is the image $c(\gamma_I)$ of $\gamma_I$ by the complex conjugation (on the modular curve), where $\alpha \in K^{\times}$ is any element of negative norm. If $\chi : \mathcal{C}(K) \rightarrow \C^{\times}$, one defines 
$$[\chi]:= \sum_{[I] \in \mathcal{C}(K)} \chi^{-1}(I) \cdot \gamma_I \in H_1(X_0(N), \Z) \otimes_{\Z} \C \text{ .}$$
Popa proved in \cite{Popa} that we have the following formula, for any character $\chi$ as above:

\begin{equation}\label{formula_of_Popa}
L(F/K, \chi, 1) =\frac{1}{\sqrt{D}}\cdot \abs* {\sum_{I \in \mathcal{C}(K)} \chi^{-1}(I)\cdot  \xi_F(\gamma_I)}^2 \text{ .}
\end{equation}

Let us consider the analogue of (\ref{equation_complex_conjug_symbol_chi}) (which was needed in the proof of Proposition \ref{prop_C_g_square}).
\begin{lem}\label{equation_complex_conjug_symbol_chi_real}
For all character $\chi$ of $\mathcal{C}(K)$, we have:
$$\overline{[\chi]} = -\chi(\mathfrak{N})\cdot W_N c([\chi])\text{ .}$$
Here, by $\overline{[\chi]}$ we mean that we apply the complex conjugation to the coefficients of the homology (not on the modular curve itself), while $c$ denotes the complex conjugation acting on $X_0(N)$.
\end{lem}
\begin{proof}
We follow \cite[\S 1.4]{Darmon_Heegner}. Recall that we have fixed a choice of $\delta_N \in \Z$ such that $\delta_N^2 \equiv D \text{ (modulo }N\text{)}$. The corresponding ideal $\mathfrak{N}$ of $\mathcal{O}_K$ is given by $(N, \frac{-\delta_N + \sqrt{D}}{2})$.
Any ideal class in $\mathcal{C}(K)$ has a representative $I$ of the form $(A, \frac{-B+\sqrt{D}}{2})$ where $N \mid \mid A$ and $B\equiv \delta_N \text{ (modulo } N\text{)}$.  We let $W_N(I)$ be the ideal given by $(NC, \frac{-B+\sqrt{D}}{2})$. With such a $I$, by definition $\gamma_I$ corresponds to the homology class of the path from $z$ to $M_Iz$ in $X_0(N)$, where $M_I := \begin{pmatrix} u-Bv & -2Cv \\ 2Av & u+Bv \end{pmatrix} \in \Gamma_0(N)$, $z$ is in the upper-half plane and $u+\sqrt{D}v$ is a choice of fundamental unit of $K$ (fixed once and for all). 

The cycle $W_N \gamma_I$ corresponds to the matrix $$\begin{pmatrix} 0 & -1 \\ N & 0 \end{pmatrix} M_I \begin{pmatrix} 0 & -1 \\ N & 0 \end{pmatrix}^{-1} = \begin{pmatrix} u+Bv & -2\frac{A}{N}v \\ 2NCv & u-Bv \end{pmatrix} = \begin{pmatrix} -1 & 0 \\ 0 & 1 \end{pmatrix} M_{W_N(I)}^{-1} \begin{pmatrix} -1 & 0 \\ 0 & 1 \end{pmatrix}^{-1} \text{ .}$$
Thus, we have $W_N \gamma_I = -c(\gamma_{W_N(I)})$.

By \cite[Proposition 1.11]{Darmon_Heegner}, for all $[I] \in \mathcal{C}_K$, we have $\gamma_{\mathfrak{N}I^{-1}} = \gamma_{W_N(I)}$. We thus have $$\gamma_{\mathfrak{N}I^{-1}} = -c W_N \gamma_I \text{ .}$$ Therefore, we have:
\begin{align*}
\overline{[\chi]} &= \sum_{I \in \mathcal{C}(K)} \chi(I) \cdot \gamma_I 
\\& =  \sum_{I \in \mathcal{C}(K)} \chi(I)^{-1} \cdot \gamma_{I^{-1}}
\\& =  \chi(\mathfrak{N}) \cdot \sum_{I \in \mathcal{C}(K)} \chi(I)^{-1} \cdot \gamma_{\mathfrak{N}I^{-1}}
\\& =- \chi(\mathfrak{N}) \cdot cW_N\left( \sum_{I \in \mathcal{C}(K)} \chi(I)^{-1} \cdot \gamma_I\right)
\\& = -\chi(\mathfrak{N}) \cdot W_N c([\chi]) \text{ .}
\end{align*}
\end{proof}

\begin{rem}\label{rem_proof_square_real}
Using (\ref{formula_of_Popa}) and Lemma \ref{equation_complex_conjug_symbol_chi_real}, the proof of Proposition \ref{prop_C_g_square}) when $K$ is real is very similar to the proof when $K$ is imaginary (we leave the details to the reader).
\end{rem}

By \cite[(109)]{DHRV}, we have
\begin{equation}\label{real_log_u}
\kappa_1^+(\gamma_I) = -\log(u(K)) \text{ ,}
\end{equation}

where the $\log$ is applied to the reduction of $u(K)$ modulo $\mathfrak{N}$. As in \cite[Proposition 5.8]{DHRV}, we have 

\begin{equation}\label{real_eq_psi_alg}
\kappa_0^-([\psi]) = (1-\psi(\mathfrak{N})^{-1})\cdot L_{\text{alg}}(\psi) \in R \text{ ,}
\end{equation}
where $L_{\text{alg}}(\psi) = \frac{12\sqrt{D}}{\pi^2}\cdot L(\psi, 1)$ belongs to $R$. Let us now examine (\ref{formula_of_Popa}) for $\chi = \mathbb{1}$ and $\chi = \psi$.

If $\chi = \mathbb{1}$, then 
\begin{align*}
\sum_{I \in \mathcal{C}(K)} \chi^{-1}(I)\cdot  \xi_F^-(\gamma_I) &= \sum_{I \in \mathcal{C}(K)} \xi_F^-(\gamma_I) 
\\& = \sum_{I \in \mathcal{C}(K)}  \xi_F^-(\gamma_{\alpha \cdot I})
\\& = -\sum_{I \in \mathcal{C}(K)}  \xi_F^-(\gamma_{I})
\end{align*}
We thus have $\sum_{I \in \mathcal{C}(K)}  \xi_F^-(\gamma_{I}) = 0$ and $\sum_{I \in \mathcal{C}(K)} \xi_F(\gamma_I) = \sum_{I \in \mathcal{C}(K)}  \xi_F^+(\gamma_I)$. Thus, we get:
\begin{equation}\label{real_key_trivial_char}
\sqrt{D}\cdot \frac{L(F/K,  1)}{(\Omega_F^+)^2} =  \abs*{ \sum_{I \in \mathcal{C}(K)}  e_F(\gamma_I) }^2 \equiv (h(K)\cdot \log(u(K)))^2  \text{ (modulo }\pi\text{),}
\end{equation}
where the last congruence follows from (\ref{real_log_u}). We shall give in \cite{Lecouturier_Wang_Classgroup} another proof of that congruence without using Popa's formula, but instead relying on Sharifi's conjecture.  

Let us now consider the case $\chi = \psi$. Note that $\psi$ has signature $(-1,-1)$, and thus $\psi(\alpha \cdot I) = - \psi(I) $ where $\alpha \in K^{\times}$ has negative norm. We have:
\begin{align*}
\sum_{I \in \mathcal{C}(K)} \psi^{-1}(I)\cdot  \xi_F^+(\gamma_I) &= \sum_{I \in \mathcal{C}(K)} \psi^{-1}(\alpha \cdot I) \cdot \xi_F^+(\gamma_{\alpha\cdot I}) 
\\& = -  \sum_{I \in \mathcal{C}(K)}  \psi^{-1}(I)\cdot  \xi_F^+(\overline{\gamma_{I}})
\\& = -\sum_{I \in \mathcal{C}(K)}  \psi^{-1}(I)\cdot  \xi_F^+(\gamma_{I})
\end{align*}
so we get $\sum_{I \in \mathcal{C}(K)} \psi^{-1}(I)\cdot  \xi_F^+(\gamma_I)=0$. Consequently, we have:
$\sum_{I \in \mathcal{C}(K)} \psi^{-1}(I)\cdot  \xi_F(\gamma_I)=i\cdot \sum_{I \in \mathcal{C}(K)} \psi^{-1}(I)\cdot  \xi_F^{-}(\gamma_I)$. We thus get:
$$
\sqrt{D}\cdot \frac{L(F/K, \psi, 1)}{(\Omega_F^-)^2} \equiv  \abs*{ \sum_{I \in \mathcal{C}(K)}  \psi^{-1}(I) \cdot e_F'(\gamma_I) }^2 \equiv  \abs*{(1-\psi(\mathfrak{N})^{-1})\cdot L_{\text{alg}}(\psi) }^2\text{ (modulo }\pi\text{),}
$$
where the last congruence comes from (\ref{real_eq_psi_alg}). Consider the Hecke polynomial $X^2-a_N(g)X+\chi_g(N)$ of our weight one form $g$ at $N$. Let $\alpha_N$ and $\beta_N$ be its roots, ordered so that $\psi(\mathfrak{N}) = \frac{\alpha_N}{\beta_N}$. Since $\abs*{\alpha_N} = \abs*{\beta_N} = 1$, we have 
$$ \abs*{1-\psi(\mathfrak{N})^{-1}}^2 = \abs*{\alpha_N-\beta_N}^2 \text{ .}$$ 
Note that 
\begin{align*}
\abs*{\alpha_N-\beta_N}^2 &= (\alpha_N-\beta_N)\cdot (\overline{\alpha_N}-\overline{\beta_N}) \\&= (\alpha_N-\beta_N)\cdot (\frac{1}{\alpha_N}-\frac{1}{\beta_N}) \\&= -\chi_g(N)^{-1} \cdot (\alpha_N-\beta_N)^2 \text{ .}
\end{align*}
Furthermore, the number $L_{\text{alg}}(\psi)$ is real (since the same is true for $L(\psi,1)$ as $\psi$ is unramified).

We have thus proved that:
\begin{equation}\label{real_key_psi_char}
\sqrt{D}\cdot \frac{L(F/K, \psi, 1)}{(\Omega_F^-)^2} \equiv  -\chi_g(N)^{-1} \cdot (\alpha_N-\beta_N)^2\cdot L_{\text{alg}}(\psi)^2 \text{ (modulo }\pi\text{),}
\end{equation}
Note that the right hand side does not depend on the ordering of $\alpha_N$ and $\beta_N$, which is reassuring because the left hand side does not.

Combining (\ref{real_key_trivial_char}), (\ref{real_key_psi_char}) and Proposition \ref{real_periods_prop}, we get:
\begin{equation}\label{main_eq_real_quadra}
4D \cdot \langle e_F, e_F' \rangle^2\cdot  \frac{L(F/K, 1)\cdot L(F/K, \psi, 1)}{\langle F, F \rangle^2} \equiv -\chi_g(N)^{-1}\cdot \left(h(K) \cdot \log(u(K)) \cdot (\alpha_N-\beta_N) \cdot L_{\text{alg}}(\psi)\right)^2 \text{ (modulo }\pi\text{).}
  \end{equation}

By Proposition \ref{prop_C_g_square}, one can write $C_g = D\cdot \frac{A_g^2}{B_g^2}$ for some $A_g \in R$ and $B_g \in \Z$. Combining (\ref{collins_imaginary_eq}) and (\ref{main_eq_real_quadra}), we get:

\begin{align*} 4 B_g^2 \cdot \langle e_F , e_F' \rangle^2 \cdot \lambda_F(G)^2 &\equiv A_g^2 \cdot  \left(h(K) \cdot \log(u(K)) \cdot (\alpha_N-\beta_N) \cdot L_{\text{alg}}(\psi)\right)^2 
\end{align*}

By Theorem \ref{abstract_main_thm} and Theorem \ref{abstract_comparison_pairings}, we have
$$\langle e_F , e_F' \rangle \cdot \lambda_F(G) \equiv \langle \kappa_0^+, \kappa_{g_p}^- \rangle \cdot \lambda_{g_p-1}(G) \equiv -24 \cdot\langle f_{g_p-1}, \mathfrak{S}_1 \rangle \cdot \lambda_{g_p-1}(G) \equiv -24\cdot \langle G, \mathfrak{S} \rangle \text{ (modulo }\pi\text{).}$$

We have thus proved:
$$4\cdot B_g^2\cdot (-24\cdot  \langle G, \mathfrak{S}\rangle )^2 \equiv  A_g^2 \cdot  \left(h(K) \cdot \log(u(K)) \cdot (\alpha_N-\beta_N) \cdot L_{\text{alg}}(\psi)\right)^2  \text{ (modulo }\pi\text{).}$$

By \cite[Lemma 5.12]{DHRV}, one can choose $u_g \in U_g$ such that $\log(\red_N(u_g)) = h(K) \cdot \log(u(K)) \cdot (\alpha_N-\beta_N) \cdot L_{\text{alg}}(\psi)$ in $R/pR$.
In conclusion, we have proven that there exists $M_g \in R$ and $u_g \in U_g$ depending only on $g$ such that $M_g^2 \cdot \langle G , \mathfrak{S} \rangle^2 = \log(\red_N(u_g))^2$ in $R/pR$ for all $N$ and $p$ as before. This proves Theorem \ref{main_thm} when $K$ is real.

\bibliography{biblio}
\bibliographystyle{plain}
\newpage

\end{document}